\documentclass{amsart}
\usepackage{graphicx}
\usepackage{amssymb,amsmath,amsthm}
\usepackage{pdfsync}
\usepackage[latin1]{inputenc}
\usepackage{tikz}

\newtheorem{teorema}{Theorem}[section]
\newtheorem{prop}[teorema]{Proposition}
\newtheorem{coro}[teorema]{Corollary}
\newtheorem{lema}[teorema]{Lemma}

\theoremstyle{remark}
\newtheorem*{remark}{Remark}

\def\R{\mathbb R}

\begin{document}

\title{A note on Hurwitz's inequality}

\author{Juli\`a Cuf\'{\i}}
\author{Eduardo Gallego$^{*}$}
\author{Agust\'{\i} Revent\'os}

\address{Departament de Matem\`atiques \\Universitat Aut\`{o}noma de Barcelona\\ 08193 Bella\-terra, Barcelona\\Catalonia\\}

\email{jcufi@mat.uab.cat}
\email{egallego@mat.uab.cat}
\email{agusti@mat.uab.cat}

\thanks{The authors where  partially supported by grants 2014SGR289 (Generalitat de Catalunya) and MTM2015-66165-P (FEDER/Mineco).}

\thanks{$^{*}$Corresponding author}


\begin{abstract}
Given a simple closed plane curve $\Gamma$ of length $L$ enclosing a compact convex set $K$  of area $F$, Hurwitz found an upper bound for the isoperimetric deficit, namely $L^2-4\pi F\leq  \pi |F_{e}|$, where $F_{e}$ is the algebraic area enclosed by the evolute of $\Gamma$. 

In this note we improve this inequality finding strictly positive lower bounds for the deficit $\pi|F_{e}|-\Delta$, where $\Delta=L^{2}-4\pi F$. These bounds involve wether the visual angle of $\Gamma$ or the pedal curve associated to $K$ with respect to the Steiner point of $K$ or the $\mathcal{L}^{2}$ distance between  $K$ and  the Steiner disk of $K$.

For each established inequality we study when equality holds. This occurs for those compact convex sets being bounded by a curve  parallel
to an hypocycloid of $3$, $4$ or $5$ cusps or the Minkowski sum  of this kind of  sets.
\end{abstract}

\maketitle

\section{Introduction}
Let $\Gamma$ be a simple closed plane curve of length $L$ enclosing a region of area $F$. The classical isoperimetric inequality states that
\begin{equation*}\label{iso1}
L^2-4\pi F\geq 0,
\end{equation*}
with equality attained only for a circle.

In the case  that $\Gamma$ bounds a convex set $K$, Hurwitz (\cite{Hurwitz}) established a kind of reverse isoperimetric inequality, namely 
\begin{equation}\label{hwfe}
L^2-4\pi F\leq  \pi |F_{e}|,
\end{equation} 
where $F_{e}$ is the algebraic area ($F_{e}\leq 0$) enclosed by the evolute of $\Gamma$. We recall that the evolute of a curve is the envelope of its normal lines. Moreover equality holds in \eqref{hwfe} if and only if $\Gamma$ is a circle or a curve parallel to an astroid.

The goal of this note is to improve Hurwitz's inequality \eqref{hwfe} finding strictly  positive lower bounds for the \emph{Hurwitz deficit} $\pi|F_{e}|-\Delta$, where $\Delta=L^{2}-4\pi F$. These bounds involve wether the visual angle of $\Gamma$ or the pedal curve associated to $K$ with respect to the Steiner point of $K$ or the $\mathcal{L}^{2}$ distance between the support function of $K$ and the support function of the Steiner disk of $K$.

Hurwitz's inequality \eqref{hwfe} can be improved without introducing new quantities for some special compact sets. For instance, if $K$ has constant width one gets
$$
L^2-4\pi F\leq \frac49 \pi |F_{e}|,
$$
as shown in Theorem \ref{bvb}.

For the general case we prove in Theorem \ref{teo51} the inequality 
$$
\pi |F_{e}|-\Delta\geq  \frac{5}{4}L^{2}+5\int_{P\notin K} (\omega-\sin\omega-\frac{2}{3}\sin^{3}\omega)\,dP,
$$
where $\omega$ is the visual angle of $\Gamma$ from $P$, that is the angle between the tangents from $P$ to $\Gamma$, and $dP$ the area measure. For the case of constant width Theorem~\ref{teo6.4} asserts that
$$
\frac{4}{9}\pi|F_{e}|-\Delta\!\ge\! \frac{64}{9} \int_{P\notin K}\!\left(\omega \!-\!2\sin\omega+\sin 2\omega-\frac{1}{4} \sin 4\omega \!-\!\sin^{3}
\omega\right)\,dP.
$$
In both cases the quantities in the right hand side are strictly positive except when the left hand side vanishes. 

In terms of the area $A$ of the pedal curve associated to the compact strictly convex set $K$, with respect to its  Steiner point,  we prove in Theorem \ref{teo7.1}
$$
\pi |F_{e}| -\Delta \ge \frac{40}{9} \left(\pi (A-F)+\frac{2}{3} L^{2}-\frac{8}{9} \int_{P\notin K}\sin^{3}\omega\,dP\right).
$$
When $K$ has constant width we obtain  (Corollary \ref{coroaf})
$$
\pi|F_{e}|-\Delta\ge \frac{40}{9} \pi (A-F).
$$
In both cases the lower bounds for the positive Hurwitz deficit are strictly positive.
\smallskip		

For each established inequality we study when equality holds. This occurs for those compact convex sets being bounded by a curve  parallel
to an hypocycloid of $3$, $4$ or $5$ cusps or the Minkowski sum  of this kind of  sets.
\smallskip		


\section{Preliminaries} 
\subsection{Convex sets and support function}\label{cvxsets}
A set $K\subset \R^{2}$ is \emph{convex} if it contains the complete segment joining every two points in the set. We shall consider nonempty  compact convex sets. The \emph{support function} of $K$ is defined as
$$
p_{K}(u):=\mathrm{sup}\{\langle x,u\rangle\, :\, x\in K\}\qquad \mathrm{for}\quad u\in \R^{2}.
$$
For a unit vector $u\in S^{1}$ the number $p_{K}(u)$ is the signed distance of the support line to $K$ with outer normal vector $u$ from the origin. The distance is negative if and only if $u$ points into the open half-plane containing the origin (cf.~\cite{Schneider2013}). 
We shall denote by $p(\varphi)$ 
the $2\pi$-periodic function obtained by evaluating $p_{K}(u)$   on $u=(\cos\varphi,\sin\varphi)$. Note that $\partial K$ is the envelope of the one parametric family of lines given by
$$
x\cos\varphi\ +y \sin\varphi\ =p(\varphi).
$$
If the support function $p(\varphi)$ is differentiable  we can parametrize the boundary $\partial K$ by
\begin{equation}\label{paramfi}
\gamma(\varphi)=p(\varphi)N(\varphi)+p'(\varphi)N'(\varphi)
\end{equation}
where $N(\varphi)=(\cos\varphi,\sin\varphi).$
When $p$ is a $\mathcal{C}^{2}$ function the radius of curvature $\rho(\varphi)$ of $\partial K$ at the point $\gamma(\varphi)$ is given by $p(\varphi)+p''(\varphi)$. Then, convexity is equivalent to $p(\varphi)+p''(\varphi)\geq 0$. We say that a $\mathcal{C}^{2}$ support function $p$ defines a \emph{strictly convex} set if $p(\varphi)+p''(\varphi)> 0$ for every value of $\varphi$. 

It can be seen that the length $L$ of $\partial K$ is given by
$$L=\int_{0}^{2\pi}p\,d\varphi.$$
A straightforward computation shows that the area $F$ of $K$ is given by
$$F=\frac{1}{2}\int_{0}^{2\pi}p(p+p'')\,d\varphi.$$
Since $p$ is $2\pi$-periodic, integrating by parts, we get
\begin{equation}\label{genis}F=\frac{1}{2}\int_{0}^{2\pi}(p^2-p'^2)\,d\varphi.\end{equation}

In general, a one parameter family of lines
$$
x \cos t +y  \sin t=f(t),
$$
where $f$ is a differentiable function, defines a  curve in the plane. In this setting the curve is not necessarily closed nor convex. When a curve $\gamma(t)$, $a\leq t\leq b$, is defined as the envelope of a family of lines of this type, for a function $f$ of  class~$\mathcal{C}^{2}$,
we say that $f(t)$ is the \emph{generalized support function} of the curve.  The \emph{area with multiplicities} swept by the radius vector of the curve is given by
\begin{equation}\label{areag}
F=\frac{1}{2}\int_{a}^{b}f(f+f'')\,dt,
\end{equation}
as a simple computation shows.

\enlargethispage{3mm}

\medskip	

Let $p(\varphi)$ be the support function of a strictly convex set $K$. Then $p_{r}(\varphi)=p(\varphi)+r$ defines for each real $r$ a parallel  curve to $\partial K$. If the origin is in the interior of $K$ then $p$ is a strictly positive 
function. If $r>0$ the function $p_{r}$ corresponds to the outer parallel set at distance $r$. When $r<0$ the curve given by $p_{r}$ is not necessarily convex (this is the case  when $|r|>\mathrm{min}(\rho)$, $\rho$ being the radius of curvature).

The Steiner formula (see for instance \cite{Schneider2013})
$$
F_{r}=\pi r^{2}+L\, r+F
$$
gives the area $F_{r}$ of the $r$-parallel  set to $K$.
The discriminant of this polynomial is the isoperimetric deficit $L^{2}-4\pi F$. It is always strictly positive except for a circle. Thus, for every convex set $K$ there are interior parallel sets with negative area. The minimum area value is $F-L^{2}/4\pi$ and it is attained for the parallel set at distance~$-L/2\pi$. Then
\begin{equation}\label{areaxi}
L^{2}-4\pi F=-4\pi F_{-L/2\pi}=4\pi |F_{-L/2\pi}|.
\end{equation}

A special type of convex sets are those  of \emph{constant width}, that is those convex sets whose orthogonal projection on any direction have the same length $w$.
In terms  of the support function $p$ of $K$,  constant width means that $p(\varphi)+p(\varphi+\pi)=w$. Expanding $p$ in Fourier series
\begin{equation}\label{fourier}
p(\varphi)=a_{0}+\sum_{n=1}^{\infty}a_{n}\cos(n\varphi)+b_{n}\sin(n\varphi),
\end{equation}
it follows that
$$p(\varphi)+p(\varphi+\pi)=2\sum_{n=0}^{\infty}(a_{2n}\cos 2n\varphi+b_{2n}\sin 2n\varphi),$$ 
so constant width is equivalent to   $a_{n}=b_{n}=0$ for all even $n>0$.

\subsection{Hypocycloids}\label{cicloid}
Consider a curve defined by the generalized support function  
$$
p(\theta)=A\sin(B\theta),  \quad \theta\in\R
$$ 
with $B$ a positive rational number and $A>0$. 
If we define $k=2B/(B-1)$ and  $A=r(k-2)$, then  $p(\theta)$ can be written in the more convenient form
$$
p(\theta)=r(k-2)\sin\left(\frac{k}{k-2}\theta\right),\quad k>2.
$$
The envelope curve given by this generalized support function can be parametrized by  
$$
\gamma(\theta)=r(k-2)\sin\left(\frac{k}{k-2}\theta\right)N(\theta)+rk\cos\left(\frac{k}{k-2}\theta\right)N'(\theta).
$$
Putting $\theta=(k-2)t/2$ the curve $\tilde\gamma(t)=\gamma(\frac{k-2}{2}t)$ has components
$$
\left.
\begin{array}{r@{\;}c@{\;}l@{}}
x(t) & = & r(k-2)\sin\left(\frac{k}{2}t\right)\cos\left(\frac{k-2}{2}t\right)-rk\cos\left(\frac{k}{2}t\right)\sin\left(\frac{k-2}{2}t\right) \\*[10pt]
y(t) & = & r(k-2)\sin\left(\frac{k}{2}t\right)\sin\left(\frac{k-2}{2}t\right)+rk\cos\left(\frac{k}{2}t\right)\cos\left(\frac{k-2}{2}t\right)
\end{array}
\right\}.
$$
Using known trigonometric identities we get
$$
\left.
\begin{array}{r@{\;}c@{\;}l@{}}
x(t) & = & r(k-1)\sin(t)-r\sin((k-1)t) \\ 
y(t) & = & r(k-1)\cos(t)+r\cos((k-1)t)
\end{array}
\right\}.
$$

This is just the parametrization of an \emph{hypocycloid} obtained by rolling a circle of radius $r$ inside a circle of radius $R=kr$.
\smallskip	

Writing $k=m/n$ with $m,n$ coprime numbers, in order to obtain a closed hypocycloid  the parameter $t$ has to vary in the interval $[0,2n\pi]$ and the parameter $\theta$ has to vary in the interval  $[0,(m-2n)\pi]$. Note that for a generalized support function~$\sin(B\theta)$ with $B$ an integer greater or equal than two, the hypocycloid is traveled twice if $B$ is odd and once if $B$ is even.
\smallskip	

When $k$ is an integer the curve has $k$ \emph{cusps} (extremal points of the curvature). For $k=m/n$ with $m$, $n$~coprime numbers the curve has $m$ cusps. In the special case $k=3$ the hypocycloid  is called a \emph{deltoid} or \emph{Steiner curve}; for $k=4$ it is called an \emph{astroid}.
\begin{figure}[h] 
   \centering
   \includegraphics[width=\textwidth]{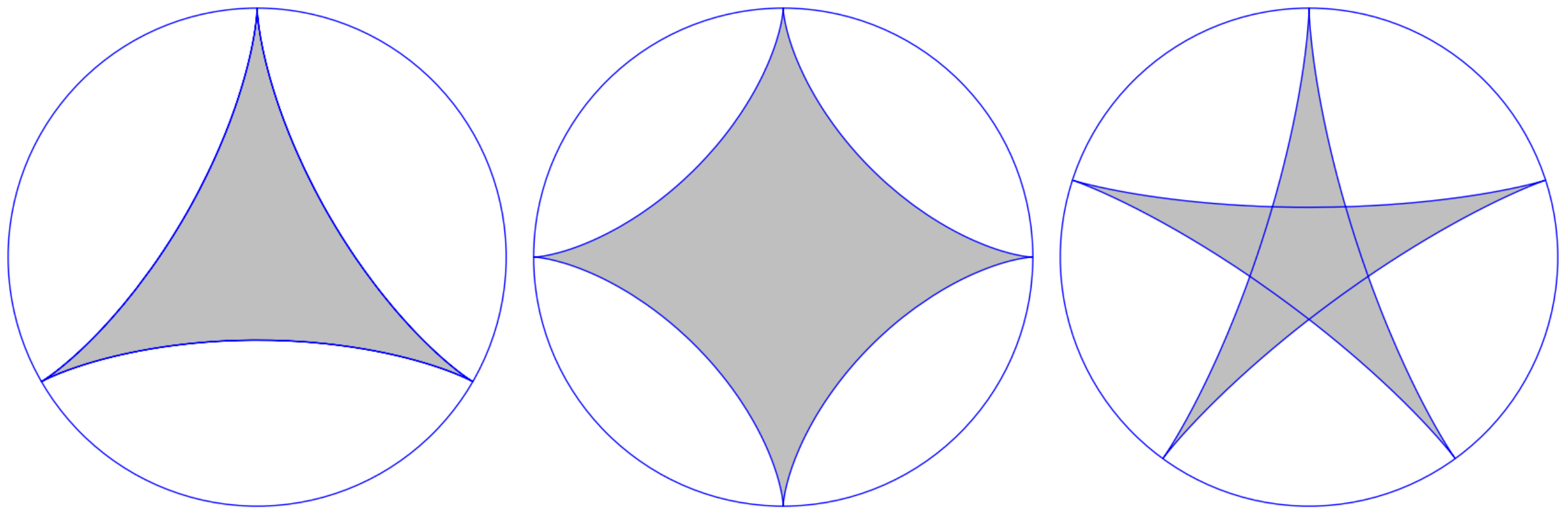} 
   \caption{Hypocycloids with $k=3$, $4$ and $5/2$.}
   \label{fig:cicloides}
\end{figure}

\subsection{Steiner point and pedal curve}
Given a compact convex set $K$ with support function $p(\varphi)$ 
the \emph{Steiner point} of $K$ is defined by the vector-valued integral
$$
s(K):=\frac{1}{\pi}\int_{0}^{2\pi}p(\varphi)N(\varphi)\,d\varphi.
$$
This functional on the space of convex sets is additive with respect to the Minkowski sum. The Steiner point is rigid motion equivariant; this means that $s(gK)=gs(K)$ for every rigid motion $g$. We remark that $s(K)$ can be considered, in the ${\mathcal C}^2$ case, as the centroid with respect to the curvature measure in the boundary $\partial K$; also we have that $s(K)$ lies in the interior of $K$ (see \cite{groemer}). 
In terms of the Fourier coefficients of $p(\varphi)$ given in \eqref{fourier} 
the Steiner point is $$s(K)=(a_{1},b_{1}).$$

The relation between the support function $p(\varphi)$ of a convex set $K$ and the support function $q(\varphi)$ of the same convex set  but with respect to a new reference with origin at the point $(a,b)$, and axes parallel to the previous 
$x$ and $y$-axes,  is given by 
$$q(\varphi)=p(\varphi)-a\cos\varphi-b\sin\varphi.$$ 

Hence, taking the Steiner point as a new origin, we have
$$q(\varphi)=a_{0}+\sum_{n\geq 2}a_{n}\cos n\varphi+b_{n}\sin n\varphi.$$

We recall that the \emph{Steiner disk} of $K$ is the disk whose center is the Steiner point and whose diameter is the mean width of $K$.
\medskip	

The associated \emph{pedal curve} to $K$ is the curve that in polar coordinates with respect to the origin is given by $r=p(\varphi)$. Notice that this curve depends on the center point from which the support function is considered. In fact it is the geometrical locus of the orthogonal projection of the center on the tangents to the curve. The area enclosed by the pedal curve is
$$
A= \frac12 \int_{0}^{2\pi}p(\varphi)^{2}\,d\varphi.
$$

\section{Hurwitz's inequality}
For a ${\mathcal C}^{1}$ function  $q(\varphi)$ of period $2\pi$, let us introduce the \emph{Wirtinger deficit} $W_{q}$ of~$q$ by
$$W_{q}=\int_{0}^{2\pi}(q'^{2}-q^{2})\,d\varphi.$$
Note that by \eqref{areag}, $W_{q}=-2F$ where $F$ is the area with multiplicities enclosed by the curve defined by the generalized support function $q$.

Recall that Wirtinger's inequality (see \cite{inequalities}) states that if 
$$
\int_{0}^{2\pi}q(\varphi)\,d\varphi=0,
$$
then
$$
W_{q}\geq 0.
$$
In particular we always have $W_{q'}\geq 0. $
\smallskip	

Now we  give a relationship between the Wirtinger deficit and  Hurwitz's deficit.

\begin{prop}\label{prop31} 
Let $K$ be a compact
strictly convex set of area $F$ bounded by a curve $\Gamma=\partial K$ of class ${\mathcal C}^{2}$ and length $L$. Let $p$ be the support function of $K$ and let $F_{e}$ be the area with multiplicities enclosed by the evolute of $\Gamma$. Then
$$\pi |F_{e}|-\Delta=\frac{\pi}{2}(W_{q'}-4W_{q})$$
where $q(\varphi)=p(\varphi)-L/2\pi$ and $\Delta=L^{2}-4\pi  F$.
\end{prop}

\begin{proof}
First of all we claim that the generalized support function for the evolute of $\Gamma$ is 
$p'(\varphi-\pi/2)$.
In fact, if the curve $\Gamma$ is parametrized by $\gamma(\varphi)$ as in \eqref{paramfi}, its evolute  can be parametrized by
\begin{equation*}
\begin{split}
\tilde\gamma(\varphi)&=\gamma(\varphi)-(p(\varphi)+p''(\varphi))N(\varphi)=p'(\varphi)N'(\varphi)-p''(\varphi)N(\varphi)\\*[5pt]
&= p'(\varphi)N\left(\varphi+\frac{\pi}{2}\right)+p''(\varphi)N'\left(\varphi+\frac{\pi}{2}\right),
\end{split}
\end{equation*}
and this proves the claim. 
So  $W_{q'}=-2F_{e}$ and since  $W_{q'}\geq 0$ we get $F_{e}\leq 0$. 

Now by \eqref{areaxi} we have that $\Delta=4\pi|F_{-L/2\pi}|$ and by \eqref{areag} we know that $W_{q}=-2F_{-L/2\pi}.$ Therefore
$$
\pi|F_{e}|-\Delta=\pi(|F_{e}|-4|F_{-L/2\pi}|)=\frac{\pi}{2}(W_{q'}-4W_{q}).
$$
%
%
\end{proof} 

\begin{remark}
Let $F$ be the area  enclosed by the curve with generalized support function the $2\pi$-periodic function $q$. As well be $F_{e}$ the area  enclosed by the evolute of this curve.
The equalities $W_{q'}=-2F_{e}$ and $W_{q}=-2F$  give 
$$\frac{1}{2}(W_{q'}-W_{q})=F-F_{e},$$ 
both areas counted with multiplicities.
Thus, for closed curves with positive curvature, we have
\begin{equation}\label{aab}
F-F_{e}=\frac{1}{2}\int_{0}^{2\pi}(q+q'')^{2}\,d\varphi=\frac{1}{2}\int_{0}^{L}\rho\,ds
\end{equation} 
where $\rho=q+q''$ is the radius of curvature and $L$
the length of the curve. We have used the relation $ds=\rho\,d\varphi$.
Equality \eqref{aab} for the case of simple closed curves that bound a strictly convex domain was proved in \cite{Hurwitz} and \cite{ERev}.\end{remark}
\smallskip	

Next Lemma compares the Wirtinger deficit of a given function with that of its derivative. The proof follows the standard pattern of the proof of Wirtinger inequality using Fourier series.
 
\begin{lema}\label{lema} Let $q=q(\varphi)$ a $2\pi$-periodic ${\mathcal C}^2$ function. Then 
$$W_{q'}\geq 4W_{q}+\frac{2}{\pi}\left(\int_{0}^{2\pi}q\,d\varphi\right)^{2}\geq 0.$$
Moreover the first inequality is an equality if and only if
$$q(\varphi)=a_{0}+a_{1}\cos\varphi+b_{1}\sin\varphi+a_{2}\cos2\varphi+b_{2}\sin2\varphi,$$ for some constants $a_{0}, a_{1},b_{1},a_{2},b_{2}\in\R$.
\end{lema}

\begin{proof}
Let $$q(\varphi)=a_{0}+\sum_{n=1}^{\infty}a_{n}\cos n\varphi+b_{n}\sin n\varphi$$
be the Fourier series expansion of $q(\varphi)$. Using the Parseval identity we get 
\begin{equation*}
\begin{split}
W_{q'}&=\pi\sum_{n=1}^{\infty} n^2(n^2-1)(a_{n}^2+b_{n}^2)\\*[5pt]
&\geq 4\pi\sum_{n=1}^{\infty} (n^2-1)(a_{n}^2+b_{n}^2)=4W_{q}+\frac{2}{\pi}\left(\int_{0}^{2\pi}q\,d\varphi\right)^{2}.
\end{split}
\end{equation*}
Equality holds if and only if  $a_{n}=b_{n}=0$, if $n\geq 3$. 
\end{proof} 
Remark that the first inequality in Lemma \ref{lema} improves Wirtinger's  inequality for the derivative of $2\pi$-periodic functions. 
\medskip	

For reader's convenience we provide a simple proof of Hurwitz's inequality based on Proposition \ref{prop31}. 

\begin{teorema}[\bf Hurwitz]\label{teo33}
Let $K$ be a compact strictly convex set of area $F$ bounded by a curve $\Gamma=\partial K$ of class ${\mathcal C}^{2}$ and length $L$ and let $F_{e}$ be the area with multiplicities enclosed by the evolute of  $\Gamma$. 
Then
\begin{equation}\label{aa}L^{2}-4\pi F\leq \pi |F_{e}|.\end{equation}
Equality holds if and only if~ $\Gamma$ is a circle or it is a curve  parallel to an astroid  at distance $L/2\pi$. 
\end{teorema}

\begin{proof}
The inequality follows from Proposition \ref{prop31} and Lemma \ref{lema}. 

Since $q(\varphi)=p(\varphi)-L/2\pi$ it is
$$
\int_{0}^{2\pi}q(\varphi)\,d\varphi=0,
$$
and so equality in \eqref{aa}
is equivalent to equality in the first inequality of Lemma \ref{lema}. This implies  
$$
p(\varphi)=a_{0}+a_{1}\cos\varphi+b_{1}\sin\varphi+a_{2}\cos2\varphi+b_{2}\sin2\varphi.
$$
Taking the Steiner point $(a_{1},b_{1})$ as a new origin of coordinates
the new support function of $K$ becomes
$$
\tilde{p}(\varphi)=a_{0}+a_{2}\cos2\varphi+b_{2}\sin2\varphi.
$$
If $a_{2}=b_{2}=0$ we get a circle. Otherwise we put $u=\varphi-\varphi_{0}+\pi/4$, where 
$$\tan 2\varphi_{0}=\frac{b_{2}}{a_{2}}$$ and in terms of  $u$ the support function of $K$ is 
$$\tilde{p}(u)=a_{0}+a \sin 2u$$
with $a=\sqrt{a_{2}^{2}+b_{2}^{2}}>0$. 
Notice that, since $\tilde p+\tilde p''>0$, one has $a<a_{0}/3=L/6\pi.$ 

From subsection~\ref{cicloid} it follows  that $\Gamma$ is parallel to an astroid at distance $a_{0}=L/2\pi$.
\end{proof}

\section{Lower bounds for Hurwitz's deficit in terms of the visual angle}
We proceed now to find a lower bound for the Hurwitz deficit $\pi|F_{e}|-\Delta$ so improving Theorem \ref{teo33}. 
If
$$p(\varphi)=a_{0}+\sum_{n\geq 1}a_{n}\cos n\varphi+b_{n}\sin n\varphi$$
is the Fourier series  of the support function of a compact  convex set $K$,  it is known that the quantities $c_{n}^{2}=a_{n}^{2}+b_{n}^{2}$, for $n\geq 2$,  are invariants  under the group of plane motions.  
This invariance will be clear through formula \eqref{ss} due to Hurwitz. 

Consider $\omega$  the \emph{visual angle} of $\Gamma$ from $P$, that is the angle between the tangents from $P$ to $\Gamma$, and let $dP$ be the area measure.
Writing 
$$I_{n}=\int_{P\notin K}\left(-2\sin(\omega)+\frac{n+1}{n-1}\sin(n-1)\omega-\frac{n-1}{n+1}\sin(n+1)\omega\right)\,dP, $$
it is proved in \cite{Hurwitz}\footnote{There is a misprint with the sign in Hurwitz's paper. 
Moreover the $c_{n}$ coefficients appearing in \eqref{ss} are different from those in Hurwitz's paper because the latter correspond to the  Fourier series of the curvature radius function.} that
\begin{equation}\label{ss}I_{n}=L^{2}+(-1)^{n}\pi^{2}(n^{2}-1)c_{n}^{2},\end{equation} $L$ being the length of the boundary of $K$.
\medskip		

For instance, if $n=2$ one gets 
\begin{equation}\label{w3}\frac{4}{3}\int_{P\not \in K}\sin^{3}\omega\, dP=L^{2}+3\pi^{2}c_{2}^{2}.\end{equation}
Moreover, this visual angle also verifies
the Crofton formula (see \cite{Hurwitz})
\begin{equation}\label{bb}
\frac{L^{2}}{2}-\pi\, F =\int_{P\not \in K} (\omega-\sin\omega) \,dP,
\end{equation}

We can prove now the following result.
\begin{teorema}\label{teo51}
Let $K$ be a compact
 strictly convex set of area $F$ bounded by a curve $\Gamma=\partial K$ of class ${\mathcal C}^{2}$ and length $L$.
Let $F_{e}$ be the  area with multiplicities enclosed by the evolute of~ $\Gamma$ and  let $\Delta$ be the isoperimetric deficit. Then
\begin{equation}\label{delta1}
\pi |F_{e}|-\Delta\geq  \frac{5}{4}L^{2}+5\int_{P\notin K} (\omega-\sin\omega-\frac{2}{3}\sin^{3}\omega)\,dP.
\end{equation}
The right hand side of this inequality is a strictly positive quantity except when $\pi|F_{e}|-\Delta=0$ in which case it also vanishes.
\end{teorema}
\begin{proof}
As we have seen in the proof of Proposition \ref{prop31}
we have
\begin{equation*}\pi |F_{e}|-\Delta=\frac{\pi}{2}(W_{q'}-4W_{q})=
 \frac{\pi}{2}\left(4\int_{0}^{2\pi}q^2d\varphi-5\int_{0}^{2\pi}q'^2\,d\varphi+\int_{0}^{2\pi}q''^2\,d\varphi\right),
\end{equation*}
where $q(\varphi)=p(\varphi)-L/2\pi$, and $p(\varphi)$
is the support function of $K$ with respect to the Steiner point.

In terms of the Fourier coefficients of $p$ 
\begin{equation}\label{6febrer}
\pi |F_{e}|-\Delta =\frac{\pi^{2}}{2}\sum_{n\geq 3}(n^{4}-5n^{2}+4)c_{n}^{2}.\end{equation}
Observe now that, for $n\geq 3$, we have $n^{4}-5n^{2}+4\geq 5(n^{2}-1)$, with equality only for $n=3$. Therefore
\begin{equation}\label{tt}
\begin{split}
\pi |F_{e}|-\Delta&\geq \frac{5\pi^{2}}{2}\sum_{n\geq 3}(n^{2}-1)c_{n}^{2}=\frac{5\pi^{2}}{2}\left(\sum_{n\geq 2}(n^{2}-1)c_{n}^{2}-3c_{2}^{2}\right)\\*[5pt]
&=\frac{5}{4}L^{2}-5\pi F-\frac{15\pi^{2}}{2}c_{2}^{2}=\frac{15}{4}L^{2}-5\pi F-\frac{10}{3}\int_{P\notin K}\sin^{3}\omega\,dP.
\end{split}
\end{equation}
Using   Crofton's formula \eqref{bb}, the last expression can be written as
$$\frac{5}{4}L^{2}+5\int_{P\not \in K} (\omega-\sin\omega-\frac{2}{3}\sin^{3}\omega)\,dP$$
and the  inequality in the theorem is proved.  
Moreover, the  sum 
$\sum_{n\geq 3}(n^{2}-1)c_{n}^{2}$
in \eqref{tt} vanishes if and only if $c_{n}=0$ for $n\geq 3$ as well as $\pi|F_{e}|-\Delta.$
\end{proof} 
We study now when equality holds in Theorem \ref{teo51}. 

\begin{prop}\label{coro2}
Equality in \eqref{delta1}  holds if and only if for the compact strictly convex set $K$  one of the following assertions holds:
\begin{itemize}
\item[a)] $K$ is a  disk or it is bounded by a curve parallel to an astroid.
\item[b)] $K$ is bounded by a curve parallel to a Steiner curve.
\item[c)] $K$ is the Minkowski sum of compact sets  of the above types.
\end{itemize}
\end{prop}

\begin{proof}
It follows from the proof of Theorem \ref{teo51}  that equality in \eqref{delta1}  holds  if and only if the support function of the domain with respect to the Steiner point is of the form
$$p(\varphi)=a_{0}+a_{2}\cos 2\varphi+b_{2}\sin 2\varphi+a_{3}\cos 3\varphi+b_{3}\sin 3\varphi.$$ 

If we put $p_{1}(\varphi)=a_{0}+a_{2}\cos 2\varphi+b_{2}\sin 2\varphi$ 
and $p_{2}(\varphi)=a_{3}\cos 3\varphi+b_{3}\sin 3\varphi$, 
we have $p(\varphi)=p_{1}(\varphi)+p_{2}(\varphi)$
and so $K$ is the Minkowski
sum of the non necessarily convex domains $D_{1}$ and $D_{2}$ with generalized support functions $p_{1}(\varphi)$ and $p_{2}(\varphi)$ respectively.

We know, by the proof of Theorem \ref{teo33},  that $D_{1}$ is the interior  of a curve parallel to an astroid or a disc.
For $p_{2}(\varphi)$ we make the change of variable given by $u=\varphi-{\varphi_{0}}/{3}$,  where $\tan \varphi_{0} ={b_{3}}/{a_{3}}$ and we get $p_{2}(u)=a\cos(3u),$ with $a=a_{3}/\cos(\varphi_{0}).$ From subsection \ref{cicloid} it follows that $D_{2}$
is the interior of a Steiner curve.  
\end{proof}

\subsection*{Relationship with the pedal curve }
\medskip
If $F$ is the area  of $K$ and $A$ is the area enclosed by the pedal curve associated to $K$ with respect to its Steiner point we obviously have $A\ge F$, with equality if and only if $K$ is a disk, and
$$
A-F=\frac{1}{2}\int_{0}^{2\pi}p'^{2}\,d\varphi.
$$

\begin{teorema}\label{teo7.1}
Let $K$ be a compact
strictly convex set of area $F$ bounded by a curve~$\Gamma=\partial K$ of class ${\mathcal C}^{2}$ and length $L$. Let  $F_{e}$ be the area with multiplicities enclosed by the evolute of~$\Gamma$. 
Let $A$ be the area enclosed by the pedal curve associated to $K$ with respect to its Steiner point. Then
$$
\pi |F_{e}| -\Delta \ge \frac{40}{9} \left(\pi (A-F)+\frac{2}{3} L^{2}-\frac{8}{9} \int_{P\notin K}\sin^{3}\omega\,dP\right).
$$
The right hand side of this inequality is a strictly positive quantity except when $\pi|F_{e}|-\Delta=0$ in which case it also vanishes.
Equality holds for the same compact sets as in Proposition~\ref{coro2}.
\end{teorema}

\begin{proof}
From \eqref{6febrer} and \eqref{w3} it follows
\begin{equation*}
\begin{split}
\pi |F_{e}| -\Delta &=\frac{\pi^{2}}{2} \sum_{n\ge 3} (n^{2}-1) (n^{2}-4) c_{n}^{2} \ge \frac{20}{9}\pi^{2}\sum_{n\ge 3}n^{2} c_{n}^{2} \\*[5pt]
&=\frac{20}{9} \pi^{2}\left(\sum_{n\ge 2}n^{2}c_{n}^{2} -4 c_{2}^{2}\right)=\frac{20}{9}\pi \left(\int p'^{2}\,d\varphi-4\pi c_{2}^{2}\right)\\*[5pt]
&=\frac{20}{9}\pi \left(\int p'^{2}\,d\varphi-4\pi c_{2}^{2}\right)=\frac{40}{9}\left[ \pi(A-F)+\frac{2}{3}L^{2}-\frac{8}{9}\int_{P \notin K}\!\sin^{3}\omega\,dP\right].\\*[5pt]
\end{split}
\end{equation*}
Moreover the right hand side vanishes if and only if $c_{n}=0$ for $n\geq 3$ as well as $\pi|F_{e}|-\Delta.$

Equality holds if and only if $c_{n}=0$, $n\ge 4$ and the result follows as in Proposition~\ref{coro2}.
\end{proof}

\subsection*{Relationship with the \boldmath$\mathcal{L}^{2}$ metric.}
\smallskip	

Consider now the quantity $\delta_{2}(K)$ equal to the distance in~$\mathcal{L}^{2}(S^{1})$, where $S^{1}$ is the unit circle, between the support function of~$K$ and the support function of the Steiner disk of~$K$. We have  that
$$
\delta_{2}(K)^{2}=\pi\sum_{n\ge 2}c_{n}^{2},
$$
where $c_{n}^{2}=a_{n}^{2}+b_{n}^{2}$ being $a_{n}$, $b_{n}$ the Fourier coefficients of the support function of~$K$ with respect to its Steiner point (\cite{groemer}).
Clearly the quantity $\delta_{2}(K)$ vanishes only when $K$ is a disk. 

\begin{teorema}\label{teo7.3}
Let $K$ be a compact strictly convex set of area $F$ bounded by a curve $\Gamma=\partial K$ of class ${\mathcal C}^{2}$ and length $L$. Let $F_{e}$ be the area with multiplicities of the evolute of~ $\Gamma$.
Then
$$
\pi |F_{e}|-\Delta \ge 20 \left(\pi \delta_{2} (K)^{2}+\frac{L^{2}}{3}-\frac{4}{9}\int_{P\notin K}\sin^{3}\omega\,dP\right).
$$
The right hand side of this inequality is a strictly positive quantity except when $\pi|F_{e}|-\Delta=0$ in which case it also vanishes.
Equality holds for the same compact sets as in Proposition~\ref{coro2}. 
\end{teorema}

\begin{proof}
According to \eqref{6febrer} and \eqref{w3}  we have
\begin{equation*}
\begin{split}
\pi|F_{e}|-\Delta&=\frac{\pi^{2}}{2}\sum_{n\ge 3}(n^{2}-1)(n^{2}-4)c_{n}^{2} \ge 20 \pi^{2}\sum_{n\ge 3}c_{n}^{2}\\*[5pt]
& =20\pi\left(\pi \sum_{n\ge 2}c_{n}^{2} -\pi c_{2}^{2}\right)
=20\pi  \delta_{2}(K)^{2}-20 \left(\frac{4}{9}\int_{P\notin K}\sin^{3}\omega\,dP-\frac{L^{2}}{3}\right)
\end{split}
\end{equation*}
as required. Equality holds if and only if $c_{n}=0$, $n\ge 4$.
\end{proof}

\section{Convex sets of constant width}
Although Hurwitz's inequality \eqref{aa} can not be improved 
for general convex domains, it is possible to obtain a stronger inequality for convex sets of \emph{constant width}, that is those convex sets whose orthogonal projection in any direction have the same length. In this case  we have the following result. 

\begin{teorema}\label{bvb}
Let $K$ be a compact strictly convex set of constant width and area~$F$ bounded by a curve $\Gamma=\partial K$ of class ${\mathcal C}^{2}$ and length $L$. Let $F_{e}$ be the area with multiplicities of the evolute of $\Gamma$.
Then
\begin{equation}\label{aaa}L^{2}-4\pi F\leq \frac{4}{9}\pi |F_{e}|.\end{equation}
Equality holds if and only if $\Gamma$ is a circle or  a curve  parallel to a Steiner curve at distance $L/2\pi$. 
\end{teorema}

\begin{proof}
Let $q(\varphi)=p(\varphi)-L/2\pi$ where $p(\varphi)$ is the support function of $K$.
As it has been said in  the proof of Proposition \ref{prop31} we have
$$
W_{q'}=2|F_{e}|,\quad W_{q}=2|F_{-L/2\pi}|=\frac{\Delta}{2\pi},
$$
and so
$$4\pi|F_{e}|-9\Delta=2\pi(W_{q'}-9W_{q}).$$

Since $K$ is of constant width, the Fourier series of its support function has only odd terms, see subsection \ref{cvxsets}. Following the proof  of Lemma \ref{lema} for this special case  one gets
$$
W_{q'}\geq 9W_{q}+\frac{9}{2\pi}\left(\int_{0}^{2\pi}q(\varphi)\,d\varphi\right)^{2}=9W_{q},
$$
and hence the inequality \eqref{aaa} follows. 
\smallskip	

Equality in \eqref{aaa} holds if and only if $a_{n}=b_{n}=0$, for $n\geq 5$.
This implies 
$$p(\varphi)=a_{0}+a_{1}\cos\varphi+b_{1}\sin\varphi+a_{3}\cos3\varphi+b_{3}\sin 3\varphi.$$
Taking the Steiner point $(a_{1},b_{1})$ as a new origin of coordinates the new support function of $K$ becomes 
$$\tilde{p}(\varphi)=a_{0}+a_{3}\cos 3\varphi+b_{3}\sin 3 \varphi.$$
We make, as in the proof of Proposition \ref{coro2}, the change of variable   $u=\varphi-{\varphi_{0}}/{3}$,  where $\tan \varphi_{0} ={b_{3}}/{a_{3}}.$
Then $$p(u)=a_{0}+a\cos(3u),$$
with $a={a_{3}}/{\cos\varphi_{0}}$. 
Notice that $a<a_{0}/8=L/16\pi$ because $p$ represents the support function of a strictly convex set $K$.

From aubsection~\ref{cicloid} it follows  that $\Gamma$ is a circle or a curve parallel to a Steiner curve.
\end{proof}

\pagebreak

\begin{figure}[h] 
   \centering
   \includegraphics[width=0.8\textwidth]{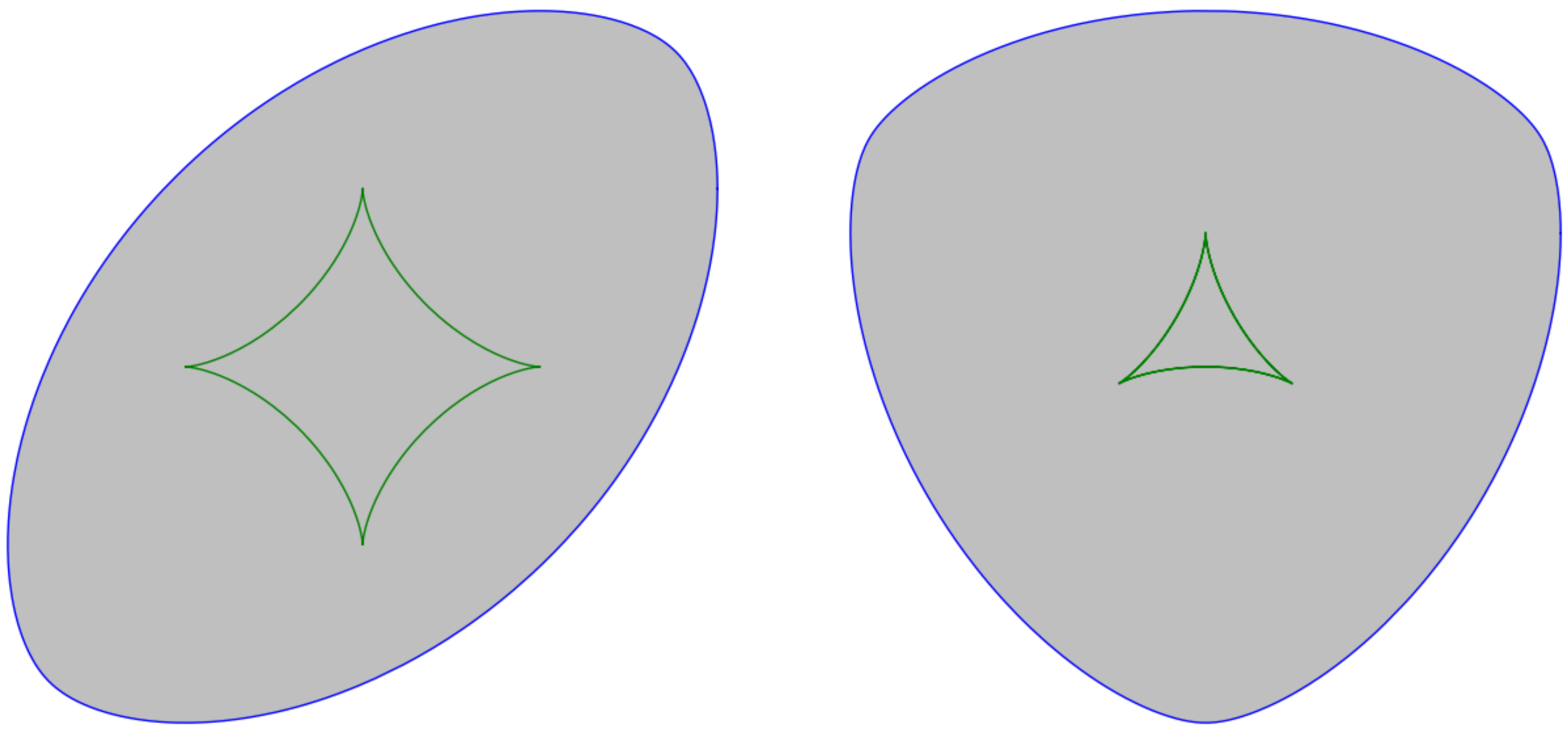} 
   \caption{Convex curves parallel to an astroid and to a Steiner curve at distance $L/2\pi$.}
   \label{fig:steiner}
\end{figure}

\begin{coro}Under the same hypothesis as in Theorem \ref{bvb} one has 
$$(A-F)\leq \frac{1}{8}|F_{e}|$$ where $A$ is the area enclosed by 
the associated pedal curve to $K$ with respect to its Steiner point. 

Equality holds if and only if $\Gamma$ is a circle or  a curve  parallel to a Steiner curve at distance $L/2\pi$.
\end{coro}

\begin{proof}
By Proposition 3.2 of \cite{CR} one has 
$$\frac{32}{9}\pi (A-F)\leq \Delta.$$
This inequality combined with \eqref{aaa} gives the result. The characterization of  equality follows from Corollary 4.4 of \cite{CR}  and Theorem \ref{bvb}.
\end{proof}

\begin{remark}
If $p$ is the support function of $K$ the \emph{Wigner caustic} of $\Gamma=\partial K$ is the curve given by the support function $q(\varphi)=\frac12(p(\varphi)-p(\varphi+\pi))$.  In \cite{zwier}    the area $A_{w}$ of the Wigner caustic of $\Gamma$ is considered. If this area is counted with multiplicities it is proved that 
\begin{equation*}
L^{2}-4\pi F\geq 4\pi |A_{w}|
\end{equation*}
with equality if and only if $K$ is of constant width.

In the case of constant width the Wigner caustic and the interior parallel curve at distance $L/2\pi$ coincide.
So using  Theorem \ref{bvb} and \eqref{areaxi} one obtains, in the case of constant width, the estimate
$$
|A_{w}|\leq \frac19 |F_{e}|
$$
with equality if and only if $\Gamma$ is a circle or a curve parallel to a Steiner curve at distance $L/2\pi$.	
\end{remark}
\medskip	

We can improve inequality \eqref{aaa} in terms of the visual angle.
\begin{teorema}\label{teo6.4}
Let $K$ be a compact
 strictly convex set of constant width and area~$F$ bounded by a curve $\Gamma=\partial K$ of class ${\mathcal C}^{2}$ and length $L$. Let $F_{e}$ be the area with multiplicities of the evolute of $\Gamma$ and $\Delta$ be the isoperimetric deficit of $\Gamma$.
Then
$$
\frac{4}{9}\pi|F_{e}|-\Delta\!\ge\! \frac{64}{9} \int_{P\notin K}\!\left(\omega \!-\!2\sin\omega+\sin 2\omega-\frac{1}{4} \sin 4\omega \!-\!\sin^{3}
\omega\right)\,dP.
$$
The right hand side of this inequality is a strictly positive quantity except when $\frac49 \pi|F_{e}|-\Delta=0$ in which case it also vanishes.

Equality  holds if and only if $K$ is a disk or it is bounded by a curve parallel  to a Steiner curve or it is bounded by a curve parallel  to an hypocycloid of five cusps or the Minkowski sum of compact sets  of the previous types.
\end{teorema}

\begin{proof}
If $p$ is the support function of $K$ we have (see the proof of Theorem~\ref{bvb})
$$
\frac{4}{9} \pi |F_{e}|-\Delta=\frac{2\pi}{9} (W_{q'}-9W_{q})=\frac{2\pi^{2}}{9}\sum_{n\ge 5}(n^{2}-1)(n^{2}-9)c_{n}^{2}
$$
with $q(\varphi)=p(\varphi)-L/2\pi$ and  $c_{n}^{2}=a_{n}^{2}+b_{n}^{2}$, being $a_{n}$, $b_{n}$ the Fourier coefficients of the support function of $K$. Recall that since $K$ has constant width we have $c_{n}=0$, for $n$ even, $n\not=0$.

Since $(n^{2}-1)(n^{2}-9)\ge 16 (n^{2}-1)$ for $n\ge 5$ it follows that 
$$
\frac{4}{9}\pi |F_{e}|-\Delta\ge \frac{32\pi^{2}}{9}  \sum_{n\ge 5}(n^{2}-1)c_{n}^{2}=\frac{32\pi^{2}}{9} \left(\sum_{n\ge 2}(n^{2}-1)c_{n}^{2}-8c_{3}^{2}\right).
$$
But
$$
\pi\sum_{n\ge 2}(n^{2}-1)c_{n}^{2} =\int^{2\pi}_{0} (p'^{2}-p^{2})\,d\varphi +2\pi a_{0}^{2}=-2F+\frac{L^{2}}{2\pi}.
$$
So we get
\begin{equation}\label{329}
\frac{4}{9}\pi |F_{e}|-\Delta \ge \frac{32}{9} \left( -2\pi F+\frac{L^{2}}{2}-8\pi^{2}c_{3}^{2}\right).
\end{equation}
Now using Crofton's formula~\eqref{bb}, the formula~\eqref{ss} for $n=3$ and the fact that $L^{2}=I_{2}$
it follows that the second member of~\eqref{329} can be written as
\begin{multline*}
-\frac{16}{3} L^{2} +\frac{64}{9} \int_{P\notin K}(\omega-\sin\omega)\,dP +\frac{32}{9} I_{3}\\
=\frac{64}{9}\int_{P\notin K}(\omega -2\sin\omega +\sin 2\omega-\frac{1}{2} \sin 4\omega -\sin^{3}\omega)\,dP.
\end{multline*}
The right hand side of \eqref{329} vanishes if and only if $c_{n}=0$ for $n\geq 5$, as well as $\frac49 \pi|F_{e}|-\Delta.$

Moreover equality in~\eqref{329} holds if and only if $c_{n}=0$, $n\ge 7$.
If we put $p_{1}(\varphi)=a_{0}+a_{3}\cos 3\varphi+b_{3}\sin 3\varphi$ 
and $p_{2}(\varphi)=a_{5}\cos 5\varphi+b_{5}\sin 5\varphi$, 
we have $p(\varphi)=p_{1}(\varphi)+p_{2}(\varphi)$
and so $K$ is the Minkowski
sum of the domains $D_{1}$ and $D_{2}$ with generalized support functions $p_{1}(\varphi)$ and $p_{2}(\varphi)$ respectively. As seen before $D_{1}$ is parallel to a Steiner curve. 
For $D_{2}$ we can write
$$
p_{2}(\varphi)=\sqrt{a^{2}_{5}+b^{2}_{5}}\,\sin(5({\varphi_{0}/ 5}+\varphi))
$$
where $\tan\varphi_{0}=b_{5}/a_{5}.$ Then $D_{2}$ corresponds to the curve with support
function 
$$q(u)=\sqrt{a^{2}_{5}+b^{2}_{5}}\sin(5u),$$ 
which  by subsection~\ref{cicloid} is the interior of an hypocycloid of five cusps.
\end{proof}
\medskip	

Also the inequalities in Theorems \ref{teo7.1} and  \ref{teo7.3} can be improved for the case of constant width, 
as shown by the following corollaries.

\begin{coro}\label{coroaf}
Under the hypothesis of Theorem \ref{teo7.1} and assuming moreover that $K$ has constant width one has 
$$
\pi|F_{e}|-\Delta\ge \frac{40}{9} \pi (A-F).
$$
Equality  holds if and only if $K$ is a disk or it is bounded by a curve parallel  to a Steiner curve.
\end{coro}

\begin{proof}
Just note that in the constant width  case,  \eqref{w3} gives
$$
\frac43 \int_{P\notin K}\sin^{3}\omega\,dP= L^{2}
$$
and apply Theorem~\ref{teo7.1}.
\end{proof}

\begin{remark}
 A straightforward calculation involving the Fourier series of $p(\varphi)$ and $q(\varphi)$ shows that 
 $$A-F\geq |A_{w}|$$ 
 where  $A_{w}$ is the area of the Wigner caustic  counted with multiplicities, with equality in the case of constant width.   So  Theorem \ref{teo7.1} and Corollary \ref{coroaf} give  lower bounds for the Hurwitz deficit $\pi|F_{e}|-\Delta$ that involve $|A_{w}|$.
\end{remark} 

\begin{figure}[h]
\centering
\begin{tikzpicture}
    \node[anchor=south west,inner sep=0] (image) at (0,0) {\includegraphics[width=0.9\textwidth]{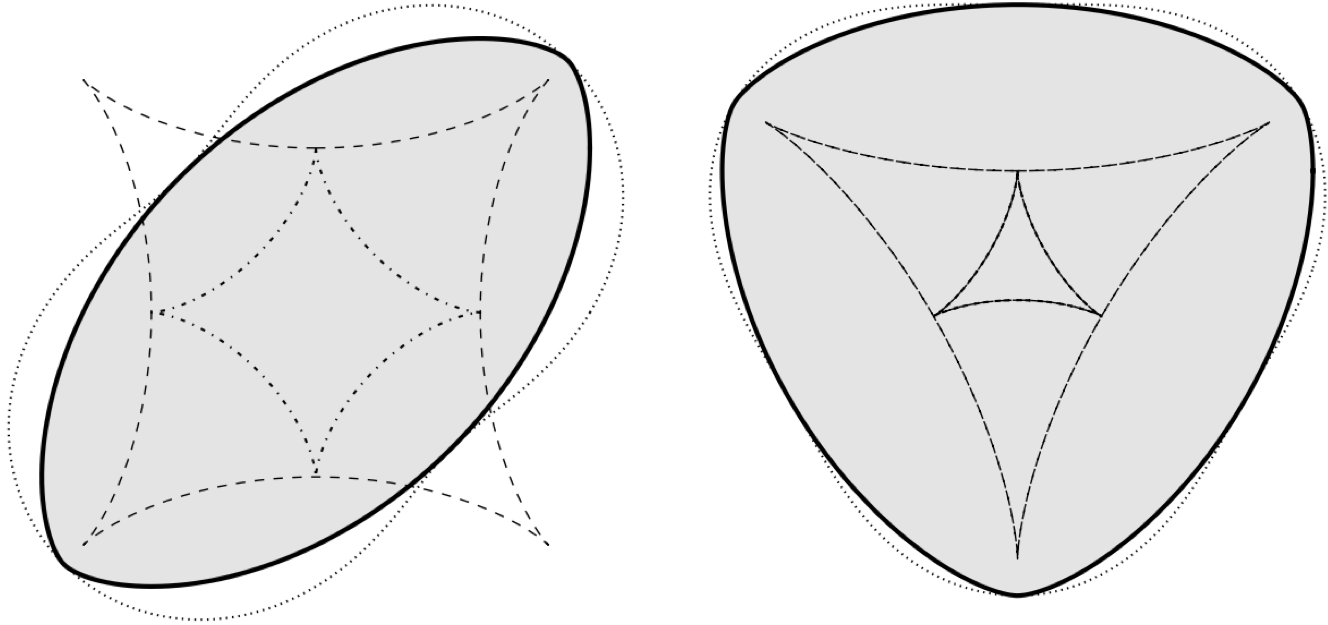}};
 \begin{scope}[x={(image.south east)},y={(image.north west)}]   
 	\node[text width=6cm, anchor=west, right] at (.1,-0.015) 	{\small \emph{Pedal}};
    \node[text width=6cm, anchor=west, right] at (0.0085,0.76) 	{\small \emph{Evolute}};
    \node[text width=6cm, anchor=west, right] at (0.26,0.65) 	{\small \emph{Parallel}};
    \node[text width=6cm, anchor=west, right] at (.63,0.07) 	{\small \emph{Pedal}};
    \node[text width=6cm, anchor=west, right] at (.58,0.6) 		{\small \emph{Evolute}};
    \node[text width=6cm, anchor=west, right] at (.78,0.6) 		{\small \emph{Wigner}$=$\emph{Parallel}};
  \end{scope}
\end{tikzpicture}
\caption{Different curves related to convex sets with central symmetry on the left and  constant width on the right.} \label{fig:corbes}
\end{figure}
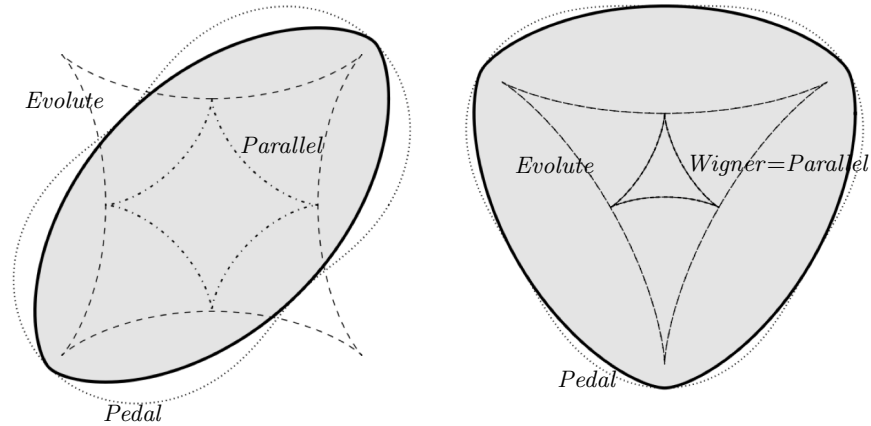

\begin{coro}Under the hypothesis of Theorem \ref{teo7.3} and assuming  that $K$ has constant width one has  
$$
\pi |F_{e}|-\Delta \ge 20\,\pi \,\delta_{2} (K)^{2}.
$$
Moreover 
$$
|F_{e}|\geq 36\, \delta_{2}(K)^{2}.
$$
Equality holds in both inequalities if and only if $K$ is a disk or it is bounded by a curve  parallel to a Steiner curve.
\end{coro}

\begin{proof} When~$K$ has constant width by \eqref{w3} one has ${L^{2}}=\frac{4}{3}\int_{P\notin K}\sin^{3}\omega\,dP$ and the first inequality follows from Theorem \ref{teo7.3}.

Then we have
\begin{equation*}
\begin{split}
\pi |F_{e}|\geq  20\,\pi \,\delta_{2} (K)^{2}+\Delta &=20\,\pi \,\delta_{2} (K)^{2}+2\pi^{2}\sum_{n\geq 3}(n^{2}-1)c_{n}^{2}\\
&\geq 20\,\pi \,\delta_{2} (K)^{2}+16\pi^{2}\sum_{n\geq 3}c_{n}^{2}\\
&=36\pi \delta_{2}(K)^{2},
\end{split}
\end{equation*}
which gives the second inequality.

Equalities hold if and only $c_{n}=0$ for $n\geq 5$.
\end{proof}

\bibliographystyle{plain}

\end{document}